\documentclass[12pt]{amsart}
\usepackage{amssymb,verbatim,amscd,amsmath,graphicx}
\usepackage{graphicx}
\usepackage{caption}
\usepackage{subcaption}
\usepackage{pdfsync}
\numberwithin{equation}{section}
\newtheorem{theorem}{Theorem}[section]
\newtheorem{lemma}[theorem]{Lemma}

\newtheorem{corollary}[theorem]{Corollary}
\newtheorem{conjecture}[theorem]{Conjecture}

\theoremstyle{definition}
\newtheorem{definition}[theorem]{Definition}
\newtheorem{def-prop}[theorem]{Definition-Proposition}
\newtheorem{remark}[theorem]{Remark}
\newtheorem{example}[theorem]{Example}

\newtheorem{problem}[theorem]{Problem}

\DeclareMathOperator{\reg}{reg}

\DeclareMathOperator{\height}{ht}

\DeclareMathOperator{\link}{link}
\DeclareMathOperator{\del}{del}
\DeclareMathOperator{\pd}{pd}

\newcommand{\ZZ}{{\mathbb Z}}
\newcommand{\NN}{{\mathbb N}}

\def\mm{{\frak m}}

\def\E{{\mathcal E}}

\def\P{{\mathcal P}}

\def\1{{\bf 1}}
\def\0{{\bf 0}}

\begin{document}

%%%%%%%%%%%%%%%%%%%%%%%%%%%%%%%%%%%%%%%%%%%%%%%%%%%%%%%%%%%%%%%%%%%%%%%%%%%%%%%%%%%

\title{Regularity of squarefree monomial ideals}

\author{Huy T\`ai H\`a}
\address{Tulane University \\ Department of Mathematics \\
6823 St. Charles Ave. \\ New Orleans, LA 70118, USA}
\email{tha@tulane.edu}
\urladdr{http://www.math.tulane.edu/$\sim$tai/}

%\keywords{hypergraphs, associated primes, monomial ideals, chromatic number, perfect graphs, Alexander duality, cover ideals}
%\subjclass[2000]{13F55, 05C17, 05C38, 05E99}
%\thanks{This work was partially supported by grants from the Simons Foundation (\#199124 to Francisco and \#202115 to Mermin). H\`a is partially supported by NSA grant H98230-11-1-0165.}

\begin{abstract}
We survey a number of recent studies of the Castelnuovo-Mumford regularity of squarefree monomial ideals. Our focus is on bounds and exact values for the regularity in terms of combinatorial data from associated simplicial complexes and/or hypergraphs.
\end{abstract}

\maketitle

\begin{quotation}
 \begin{center}
  \emph{Dedicated to Tony Geramita, a great teacher, colleague and friend.}
 \end{center}
\end{quotation}

%%%%%%%%%%%%%%%%%%%%%%%%%%%%%%%%%%%%%%%%%%%%%%%%%%%%%%%%%%%%%%%%%%%%%%%%%%%%%%

\section{Introduction} \label{sec.intro}

Castelnuovo-Mumford regularity (or simply regularity) is an important invariant in commutative algebra and algebraic geometry that governs the computational complexity of ideals, modules and sheaves. Computing or finding bounds for the regularity is a difficult problem. Many simply stated questions and conjectures have been verified only in very special cases. For instance, the Eisenbud-Goto conjecture, which states that the regularity of a projective variety is bounded by the difference between its degree and codimension, has been proved only for arithmetic Cohen-Macaulay varieties, curves and surfaces.

The class of squarefree monomial ideals is a classical object in commutative algebra, with strong connections to topology and combinatorics, which continues to inspire much of current research. During the last few decades, advances in computer technology and speed of computation have drawn many researchers' attention toward problems and questions involving this class of ideals. Investigating the regularity of squarefree monomial ideals has evolved to be a highly active research topic in combinatorial commutative algebra.

In this paper, we survey a number of recent studies on the regularity of squarefree monomial ideals. Even for this class of ideals, the list of results on regularity is too large to exhaust. Our focus will be on studies that find bounds and/or compute the regularity of squarefree monomial ideals in terms of combinatorial data from associated simplicial complexes and hypergraphs. Our aim is to provide readers with an adequate overall picture of the problems, results and techniques, and to demonstrate similarities and differences between these studies. At the same time, we hope to motivate interested researchers, and especially graduate students, to start working in this area. We shall showcase results in a timeline to exhibit (if possible) trends in developing new and/or more general results from previous ones.

The paper is structured as follows. In the next section we collect basic notation and terminology from commutative algebra and combinatorics that will be used in the paper. This section provides readers who are new to the research area the necessary background to start. More advanced readers can skip this section and go directly to Section \ref{sec.induction}. In Section \ref{sec.induction}, we outline a number of inductive results that appear to be common techniques in many recent works. Some of these results arise from simple short exact sequences in commutative algebra, while others may require heavier topological machineries to prove. Section \ref{sec.bound} is devoted to studies that provide bounds for the regularity of squarefree monomial ideals. In Section \ref{sec.compute}, we discuss studies that explicitly compute the regularity of squarefree monomial ideals, and identify classes of ideals with small regularity. In these sections (Sections \ref{sec.induction}, \ref{sec.bound} and \ref{sec.compute}), we do not generally prove results being surveyed. Rather, occasionally for initial theorems of various inductive techniques we shall sketch the proofs to demonstrate the methods used. The last section provides a number of open problems and questions that we would like to see answered.

\noindent{\bf Acknowledgement.} The author would like to thank an anonymous referee for a careful reading and many helpful comments. The author would also like to thank S.A. Seyed Fakhari for pointing out a mistake in our original definition of $H_E$ in Theorem \ref{thm.ind3}.

%%%%%%%%%%%%%%%%%%%%%%%%%%%%%%%%%%%%%%%%%%%%%%%%%%%%%%%%%%%%%%%%%%%%%%%%%%%%%

\section{Preliminaries} \label{sec.prel}

We begin by recalling familiar notation and terminology from commutative algebra and combinatorics. We follow standard texts \cite{Berge, HH, MS, Stanley} in these fields. %More advanced readers should be comfortable to skip this section and proceed directly to Section \ref{sec.induction}.

We shall investigate the connection between commutative algebra and combinatorics via the notion of Stanley-Reisner ideals and edge ideals. To illustrate a general framework, we will not include studies that are based on constructions that do not necessarily give one-to-one correspondences between squarefree monomial ideals and combinatorial structures; for instance, \emph{path ideals} of directed trees (cf. \cite{BHO}). %Also, we shall not use the language of \emph{labeled hypergraphs} introduced in \cite{LM} since they are just the \emph{dual} notion of the edge ideals/hypergraphs framework that will be used.

Throughout the paper, $K$ will denote any infinite field and $R = K[x_1, \dots, x_n]$ will be a polynomial ring over $K$. For obvious reasons, we shall identify the variables $x_1, \dots, x_n$ with the vertices of simplicial complexes and hypergraphs being discussed. By abusing notation, we also often identify a subset $V$ of the vertices $X = \{x_1, \dots, x_n\}$ with the squarefree monomial $x^V = \prod_{x \in V} x$ in the polynomial ring $R$. In examples, we use $a,b,c, \dots$ as variables instead of the $x_i$s.

\subsection{Simplicial complexes} A \emph{simplicial complex} $\Delta$ over the vertex set $X = \{x_1, \dots, x_n\}$ is a collection of subsets of $X$ such that if $F \in \Delta$ and $G \subseteq F$ then $G \in \Delta$. Elements of $\Delta$ are called \emph{faces}. Maximal faces (with respect to inclusion) are called \emph{facets}. For $F \in \Delta$, the \emph{dimension} of $F$ is defined to be $\dim F = |F|-1$. The \emph{dimension} of $\Delta$ is $\dim \Delta = \max \{\dim F ~|~ F \in \Delta\}$. The complex is called \emph{ pure} if all of its facets are of the same dimension.

Let $\Delta$ be a simplicial complex and let $Y \subseteq X$ be a subset of its vertices. The \emph{induced subcomplex} of $\Delta$ on $Y$, denoted by $\Delta[Y]$, is the simplicial complex with vertex set $Y$ and faces $\{F \in \Delta ~|~ F \subseteq Y\}$.

\begin{definition} Let $\Delta$ be a simplicial complex over the vertex set $X$, and let $\sigma \in \Delta$.
\begin{enumerate}
\item The \emph{deletion} of $\sigma$ in $\Delta$, denoted by $\del_\Delta(\sigma)$, is the simplicial complex obtained by removing $\sigma$ and all faces containing $\sigma$ from $\Delta$.
\item The \emph{link} of $\sigma$ in $\Delta$, denoted by $\link_\Delta(\sigma)$, is the simplicial complex whose faces are
$$\{ F \in \Delta ~|~ F \cap \sigma = \emptyset, \sigma \cup F \in \Delta\}.$$
\end{enumerate}
\end{definition}

\begin{definition} A simplicial complex $\Delta$ is recursively defined to be \emph{vertex decomposable} if either
\begin{enumerate}
\item $\Delta$ is a simplex; or
\item there is a vertex $v$ in $\Delta$ such that both $\link_\Delta(v)$ and $\del_\Delta(v)$ are vertex decomposable, and all facets of $\del_\Delta(v)$ are facets of $\Delta$.
\end{enumerate}
A vertex satisfying condition (2) is called a \emph{shedding vertex}, and the recursive choice of shedding vertices are called a \emph{ shedding order} of $\Delta$.
\end{definition}

Recall that a simplicial complex $\Delta$ is said to be \emph{shellable} if there exists a linear order of its facets $F_1, F_2, \dots, F_t$ such that for all $k = 2, \dots, t$, the subcomplex $\Big(\bigcup_{i=1}^{k-1} \overline{F_i}\Big) \bigcap \overline{F_k}$ is pure and of dimension $(\dim F_k - 1)$. Here $\overline{F}$ represents the simplex over the vertices of $F$. It is a celebrated fact that \emph{pure} shellable complexes give rise to \emph{Cohen-Macaulay Stanley-Reisner rings}. For more details on Cohen-Macaulay rings and modules, we refer the reader to \cite{BrHe}. The notion of Stanley-Reisner rings will be discussed later in the paper. Recall also that a ring or module is \emph{sequentially Cohen-Macaulay} if it has a filtration in which the factors are Cohen-Macaulay and their dimensions are increasing. This property corresponds to (\emph{nonpure}) shellability in general.

Vertex decomposability can be thought of as a combinatorial criterion for shellability and sequentially Cohen-Macaulayness. In particular, for a simplicial complex $\Delta$,
$$\Delta \textrm{ vertex decomposable } \Rightarrow \Delta \textrm{ shellable } \Rightarrow \Delta \textrm{ sequentially Cohen-Macaulay}.$$

\begin{figure}[h!]
\centering
\includegraphics[height=1.2in]{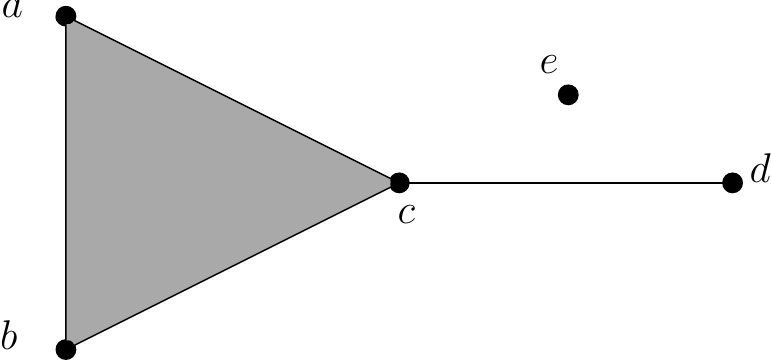}
\centering\caption{A vertex decomposable simplicial complex.}\label{fig:SimplicialComplex}
\end{figure}

\begin{example} The simplicial complex $\Delta$ in Figure \ref{fig:SimplicialComplex} is a nonpure simplicial complex of dimension 2. It has 3 facets; the facet $\{a,b,c\}$ is of dimension 2, the facet $\{c,d\}$ is of dimension 1, and the facet $\{e\}$ is of dimension 0. The complex $\Delta$ is vertex decomposable with $\{e, d\}$ as a shedding order.
\end{example}

%%%%%%%%%%%%%%%%%%%%%%%%%%%%

\subsection{Hypergraphs} A hypergraph $H = (X,\E)$ over the vertex set $X = \{x_1, \dots, x_n\}$ consists of $X$ and a collection $\E$ of nonempty subsets of $X$; these subsets are called the \emph{edges} of $H$. A hypergraph $H$ is \emph{simple} if there is no nontrivial containment between any pair of its edges. Simple hypergraphs are also referred to as \emph{clutters} or \emph{Sperner systems}. All hypergraphs considered in this paper will be simple.

When working with a hypergraph $H$, we shall use $X(H)$ and $\E(H)$ to denote its vertex and edge sets, respectively. We shall denote by $\textrm{is}(H)$ the set of (\emph{isolated}) vertices that do not belong to any edge in $H$, and let $H^{\textrm{red}}$ be the hypergraph obtained by removing vertices in $\textrm{is}(H)$ from the vertex set of $H$. An edge $\{v\}$ consisting of a single vertex is often referred to as an \emph{isolated loop} (this is not to be confused with an isolated vertex).

Let $Y \subseteq X$ be a subset of the vertices in $H$. The \emph{induced subhypergraph} of $H$ on $Y$, denoted by $H[Y]$, is the hypergraph with vertex set $Y$ and edge set $\{E \in \E ~|~ E \subseteq Y\}$. The \emph{contraction} of $H$ to $Y$ is the hypergraph with vertex set $Y$ and edges being minimal nonempty elements of $\{ E \cap Y ~|~ E \in \E\}$.

\begin{definition} Let $H$ be a simple hypergraph.
\begin{enumerate}
\item A collection $C$ of edges in $H$ is called a \emph{matching} if the edges in $C$ are pairwise disjoint. The maximum size of a matching in $H$ is called its \emph{matching number}.
\item A collection $C$ of edges in $H$ is called an \emph{induced matching} if $C$ is a matching, and $C$ consists of all edges of the induced subhypergraph $H[\cup_{E \in C} E]$ of $H$. The maximum size of an induced matching in $H$ is called its \emph{ induced matching number}.
\end{enumerate}
\end{definition}

\begin{example} Figure \ref{fig:SimplicialComplex} can be viewed as a hypergraph over the vertex set $\{a,b,c,d,e\}$ with edges $\{a,b,c\}, \{c,d\}$ and $\{e\}$. The collection $\{\{a,b,c\}, \{e\}\}$ forms an induced matching in this hypergraph.
\end{example}

A \emph{graph} is a hypergraph in which all edges are of cardinality 2. The \emph{complement} of a graph $G$, denoted by $G^c$, is the graph with the same vertex set and an edge $E$ is in $G^c$ if and only if $E$ is not in $G$.

\begin{definition} A graph $G$ is called \emph{chordal} if it has no induced cycles of length $\ge 4$.
\end{definition}

A hypergraph $H$ is \emph{$d$-uniform} if all its edges have cardinality $d$. For an edge $E$ in $H$, let $N(E) = \{x \in X ~|~ \exists F \subseteq E \textrm{ s.t. } F \cup \{x\} \in \E\}$ be the set of \emph{neighbors} of $E$, and let $N[E] = N(E) \cup E$. %We define $H_E$ to be the induced subhypergraph of $H$ over the vertex set $X \setminus N[E]$.

\begin{definition} Let $H = (X,\E)$ be a simple hypergraph. Let $E$ be an edge and let $Y$ be a subset of the vertices in $H$.
\begin{enumerate}
\item Define $H \setminus E$ to be the hypergraph obtained by deleting $E$ from the edge set of $H$. This is often referred to as the \emph{deletion} of $E$ from $H$.
\item Define $H \setminus Y$ to be the hypergraph obtained from $H$ by deleting the vertices in $Y$ and all edges containing any of those vertices.
\item Define $H_E$ to be the contraction of $H \setminus N(E)$ to $X \setminus N[E]$.
\end{enumerate}
\end{definition}
Note that when $H$ is a graph and $E$ is an edge, then $H_E$ is just the induced subgraph of $H$ over the vertex set $X \setminus N[E]$.

\begin{definition} Let $H = (X,\E)$ be a simple hypergraph.
\begin{enumerate}
\item A collection of vertices $V$ in $H$ is called an \emph{independent set} if there is no edge $E \in \E$ such that $E \subseteq V$.
\item The \emph{independence complex} of $H$, denoted by $\Delta(H)$ is the simplicial complex whose faces are independent sets in $H$.
\end{enumerate}
\end{definition}

\begin{example} The simplicial complex $\Delta$ in Figure \ref{fig:SimplicialComplex} is the independence complex of the graph in Figure \ref{fig:Graph}.
\begin{figure}[h!]
\centering
\includegraphics[height=1.7in]{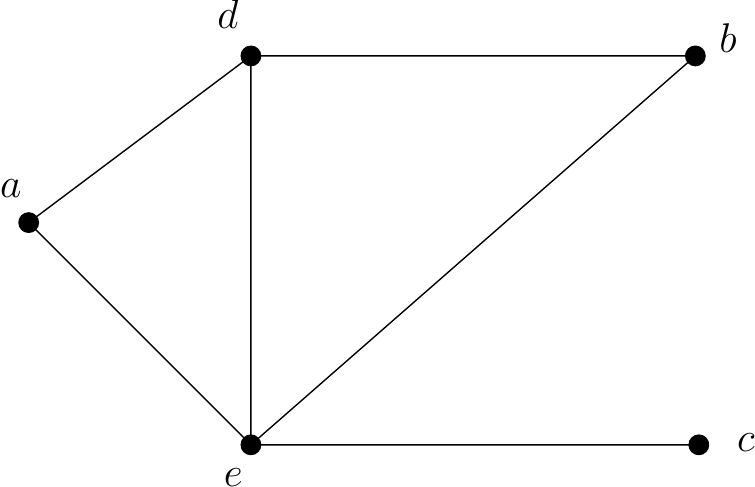}
\caption{A simple graph whose independence complex is in Figure \ref{fig:SimplicialComplex}.} \label{fig:Graph}
\end{figure}
\end{example}

\begin{remark} We call a hypergraph $H$ vertex decomposable (shellable, sequentially Cohen-Macaulay) if its independence complex $\Delta(H)$ is vertex decomposable (shellable, sequentially Cohen-Macaulay).
\end{remark}

%%%%%%%%%%%%%%%%%%%%%%%%%%%%

\subsection{Stanley-Reisner ideals and edge ideals} The Stanley-Reisner ideal and edge ideal constructions are well-studied correspondences between commutative algebra and combinatorics. Those constructions arise by identifying minimal generators of a squarefree monomial ideal with the minimal nonfaces of a simplicial complex or the edges of a simple hypergraph.

Stanley-Reisner ideals were developed in the early 1980s (cf. \cite{Stanley}) and has led to many important homological results (cf. \cite{BrHe, Peeva}).

\begin{definition} \label{def.SRideal}
Let $\Delta$ be a simplicial complex on $X$. The \emph{Stanley-Reisner ideal} of $\Delta$ is defined to be
$$I_\Delta = \big( x^F ~|~ F \subseteq X \textrm{ is not a face of } \Delta \big).$$
\end{definition}

\begin{example} Let $\Delta$ be the simplicial complex in Figure \ref{fig:SimplicialComplex} and let $R = K[a,b,c,d,e]$. Then the minimal nonfaces of $\Delta$ are $\{a,d\}, \{a,e\}, \{b,d\}, \{b,e\}, \{c,e\}$ and $\{d,e\}$. Thus,
$$I_\Delta = (ad,ae,bd,be,ce,de).$$
\end{example}

\begin{example} The simplicial complex $\Delta$ in Figure \ref{fig:ProjPlane} represents a minimal triangulation of the real projective plane. Its Stanley-Reisner ideal is
$$I_\Delta = (abc,abe,acf,ade,adf,bcd,bdf,bef,cde,cef).$$
\begin{figure}[h!]
\centering
\includegraphics[height=2.3in]{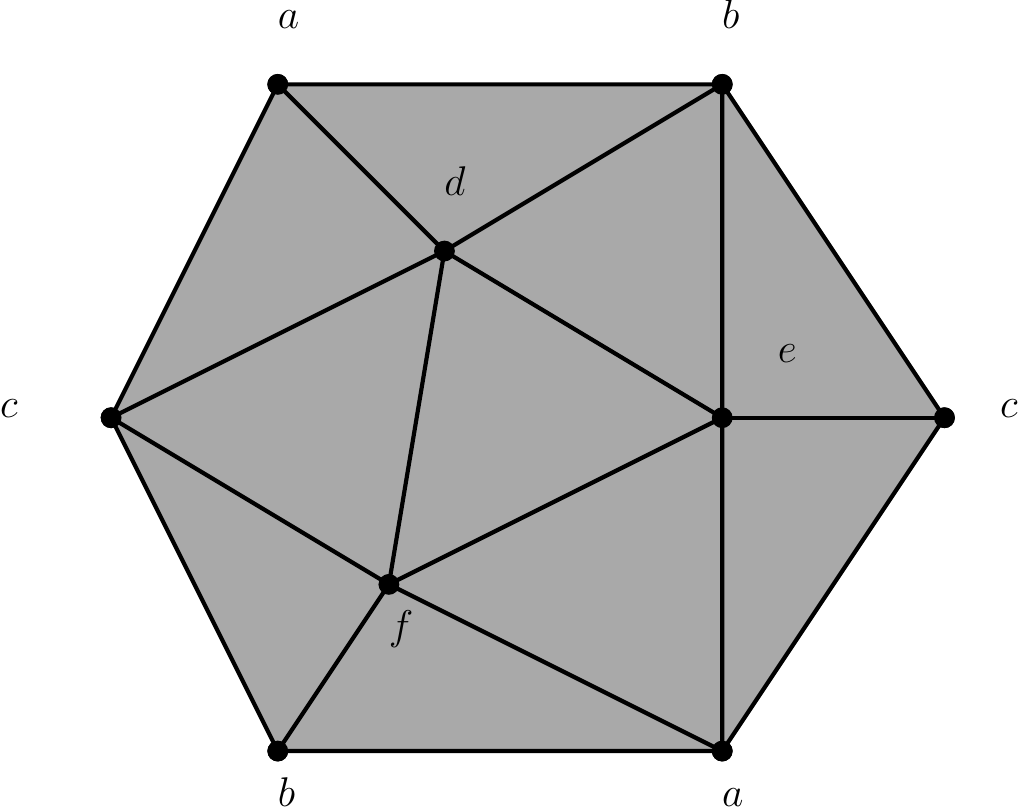}
\caption{A minimal triangulation of the real projective plane.} \label{fig:ProjPlane}
\end{figure}
\end{example}

The edge ideal construction was introduced in \cite{V} for graphs and later generalized to hypergraphs in \cite{HVT2008}. This construction is similar to that of facet ideals in \cite{Faridi}.

\begin{definition} \label{def.edgeideal}
Let $H$ be a simple hypergraph on $X$. The \emph{edge ideal} of $H$ is defined to be
$$I(H) = \big( x^E ~|~ E \subseteq X \textrm{ is an edge in } H\big).$$
\end{definition}

The notion of Stanley-Reisner ideal and edge ideal give the following one-to-one correspondences that allow us to pass back and forth from squarefree monomial ideals to simplicial complexes and simple hypergraphs.

$$\left\{ \begin{array}{ll} \textrm{simplicial complexes} \\ \textrm{over } X\end{array}\right\} \longleftrightarrow \left\{\begin{array}{ll} \textrm{squarefree monomial} \\ \textrm{ideals in } R\end{array}\right\} \longleftrightarrow \left\{\begin{array}{ll} \textrm{simple hypergraphs} \\ \textrm{over } X\end{array} \right\}$$

Stanley-Reisner ideals and edge ideals are also closely connected via the notion of independence complex.

\begin{lemma} \label{lem.IndComp}
Let $H$ be a simple hypergraph and let $\Delta = \Delta(H)$ be its independence complex. Then
$$I_\Delta = I(H).$$
\end{lemma}

\begin{example} The edge ideal of the graph $G$ in Figure \ref{fig:Graph} is the same as the Stanley-Reisner ideal of its independence complex, the simplicial complex in Figure \ref{fig:SimplicialComplex}.
\end{example}

%%%%%%%%%%%%%%%%%%%%%%%%%%%%

\subsection{Alexander duality} The Alexander duality theory for simplicial complexes carries nicely over to squarefree monomial ideals and have proved to be a significant tool in the study of these ideals.

\begin{definition} Let $\Delta$ be a simplicial complex over the vertex set $X$. The Alexander dual of $\Delta$, denoted by $\Delta^\vee$, is the simplicial complex over $X$ with faces
$$\{X \setminus F ~|~ F \not\in \Delta\}.$$
\end{definition}

Notice that $\Delta^{\vee\vee} = \Delta$. If $I = I_\Delta$ then we shall denote by $I^\vee$ the Stanley-Reisner ideal of the Alexander dual $\Delta^\vee$. Also, if $I = I(H)$ then we shall denote by $H^\vee$ the simple hypergraph corresponding to $I^\vee$.

It is a celebrated result of Terai \cite{Terai} that the regularity of a squarefree monomial ideal can be related to the projective dimension of its Alexander dual.

\begin{theorem} \label{thm.Terai}
Let $I \subseteq R$ be a squarefree monomial ideal. Then
$$\reg(I) = \pd(R/I^\vee).$$
\end{theorem}

Theorem \ref{thm.Terai} basically says that studying the regularity of squarefree monomial ideals is equivalent to studying the projective dimension of squarefree monomial ideals. Many studies in commutative algebra take the later point of view.

%%%%%%%%%%%%%%%%%%%%%%%%%%%%

%%%%%%%%%%%%%%%%%%%%%%%%%%%%

\subsection{Castelnuovo-Mumford regularity} The regularity of graded modules over the polynomial ring $R$ can be defined in various ways. Let $\mm$ denote the maximal homogeneous ideal in $R$.

\begin{definition} \label{def.reg}
Let $M$ be a finitely generated graded $R$-module. For $i \ge 0$, let
$$a^i(M) = \left\{ \begin{array}{ll} \max\left\{l \in \ZZ ~\Big|~ \big[H^i_\mm(M)\big]_l \not= 0\right\} & \textrm{if } H^i_\mm(M) \not= 0 \\ - \infty & \textrm{otherwise.} \end{array} \right.$$
The \emph{regularity} of $M$ is defined to be
$$\reg(M) = \max_{i \ge 0} \{a^i(M)\}.$$
\end{definition}

Note that $a^i(M) = 0$ for $i > \dim M$, so the regularity of $M$ is well-defined. This invariant can also be computed via the minimal free resolution (cf. \cite{Chardin:2007, EG}).

\begin{definition} \label{def-prop}
Let $M$ be a graded $R$-module and let
$$0 \rightarrow \bigoplus_{j \in \ZZ} R(-j)^{\beta_{pj}(M)} \rightarrow \cdots \rightarrow \bigoplus_{j \in \ZZ} R(-j)^{\beta_{0j}(M)} \rightarrow M \rightarrow 0$$
be its minimal free resolution. Then the regularity of $M$ is given by
$$\reg(M) = \max \{j-i ~|~ \beta_{ij}(M) \not= 0\}.$$
\end{definition}

By looking at the minimal free resolution, it is easy to see that $\reg(R/I) = \reg(I) - 1$, so we shall work with $\reg(I)$ and $\reg(R/I)$ interchangeably. %We often pass back and forth from a smaller polynomial ring to a bigger one by adding/dropping indeterminates. The following lemma allows us to do so without changing the regularity.

%\begin{lemma} \label{lem.reginquot}
%Let $I \subseteq R$ be a homogeneous ideal in a polynomial ring and let $x$ be a new variable. Let $S = R[x]$. Then
%$$\reg(R/I) = \reg(S/I+x).$$
%\end{lemma}

The focus of this paper is on squarefree monomial ideals. As we have seen, squarefree monomial ideals are in direct one-to-one correspondence with simplicial complexes, and the regularity of these ideals can also be computed from the reduced homology groups of corresponding simplicial complexes. We summarize this connection in the following lemma.

\begin{lemma} \label{lem:RegFromTopology}
For a simplicial complex $\Delta$, the following are equivalent:
\begin{enumerate}
\item $\reg(R/I_{\Delta})\geq d$.
\item $\tilde{H}_{d-1}(\Delta[S])\neq0$, where $\Delta[S]$ denotes the
induced subcomplex on some subset $S$ of vertices.
\item $\tilde{H}_{d-1}(\link_{\Delta}\sigma)\neq0$ for some face $\sigma$
of $\Delta$.
\end{enumerate}
\end{lemma}

\begin{proof} The equivalence of (1) and (2)
follows directly from Definition \ref{def-prop}, together with Hochster's formula for graded Betti numbers (as stated
in \cite[Corollary 5.12]{MS}). The equivalence of (1) and (3) follows directly from the local cohomology characterization
of regularity, together with the fact that $H_{\mathfrak{m}}^{i}(R/I_{\Delta})_{-\sigma}\cong\tilde{H}^{i-\left|\sigma\right|-1}(\link_{\Delta}\sigma)$
(see \cite[Chapter 13.2]{MS}).
\end{proof}

\begin{remark} For simplicity, if $I = I_\Delta$ then we sometimes write $\reg(\Delta)$ for $\reg(I)$, and if $I = I(H)$ then we write $\reg(H)$ for $\reg(I)$.
\end{remark}

%%%%%%%%%%%%%%%%%%%%%%%%%%%%%%%%%%%%%%%%%%%%%%%%%%%%%%%%%%%%%%%%%%%%%%%%%%%%%%

\section{Regularity and induction} \label{sec.induction}

The backbone of most of the studies that we survey is mathematical induction based on combinatorial structures of given simplicial complexes and hypergraphs. The technique is to relate the regularity of a squarefree monomial ideal corresponding to a simplicial complex and/or hypergraph to that of smaller sub-ideals corresponding to subcomplexes and/or subhypergraphs. In this section, we discuss a number of inductive results that lie in the core of most of these studies. It is worth noting that there are also inductive results that go from smaller simplicial complexes and/or hypergraphs to larger ones; for instance, in the work of \cite{BC2013,MPZ}. Since we eventually are interested in bounds or values for the regularity, those works are beyond the scope of our survey.

We start the section by a few crude inductive bounds for the regularity of a hypergraph (or simplicial complex) in terms of that of subhypergraphs (or subcomplexes), that follow directly from Lemma \ref{lem:RegFromTopology}.

\begin{lemma} \label{lem:RegShrinksInSubhypergraphs} \quad
\begin{enumerate}
\item Let $H$ be a simple hypergraph.
Then $\reg(H)\geq\reg(H')$ for any induced subhypergraph $H'$
of $H$.
\item Let $\Delta$ be a simplicial complex.
Then $\reg(\Delta)\geq\reg(\link_{\Delta}(\sigma))$ for
any face $\sigma$ of $\Delta$.
\end{enumerate}
\end{lemma}

For any homogeneous ideal $I \subseteq R$ and any homogeneous element $h \in R$ of degree $d$, the following short exact sequence is standard in commutative algebra:
\begin{align}
0 \longrightarrow \frac{R}{I:h}(-d) \stackrel{\times h}{\longrightarrow} \frac{R}{I} \longrightarrow \frac{R}{I+h} \longrightarrow 0. \label{eq.ses}
\end{align}
By taking the long exact sequence of local cohomology modules associated to (\ref{eq.ses}), we get
\begin{align}
\reg(I) \le \max \{\reg(I:h) + d, \reg(I,h)\}. \label{eq.simplereg}
\end{align}

\begin{remark} \label{rem.1of2}
If $h$ is an indeterminate of $R$ appearing in $I$, then it was shown in \cite[Lemma 2.10]{DHS} that, in fact, $\reg(I)$ is always equal to either $\reg(I:h) + 1$ or $\reg(I,h)$.
\end{remark}

In practice, induction usually starts by deleting a vertex or a face (edge). That is, $h$ is often taken to be the variable corresponding to a vertex or the product of variables corresponding to a face (edge) of the simplicial complex (hypergraph).

\begin{remark}
Let $I = I_\Delta$ be the Stanley-Reisner ideal of a simplicial complex $\Delta$, and let $h = x^\sigma$ for a face $\sigma$ of dimension $d-1$ in $\Delta$. Then $I:h = I_{\link_\Delta(\sigma)}$ and $I+h = I_{\del_\Delta(\sigma)} + h$.
%On the other hand, if $I = I(H)$ is the edge ideal of a simple hypergraph $H$ and $h = x^E$ for a subset $E$ of cardinality $d$ contained in some edge of $H$, then $I:h = I(H / E)$ and $I+h = I(H \setminus E) + h$.
\end{remark}

For a subset $V$ of the vertices in a hypergraph $H$, let $H : V$ and $H + V$ denote the hypergraphs corresponding to the squarefree monomial ideals $I(H) : x^V$ and $I(H) + x^V$, respectively. As a consequence of (\ref{eq.simplereg}), we have the following inductive bounds.

\begin{theorem} \quad \label{thm.ind11}
\begin{enumerate}
\item Let $\Delta$ be a simplicial complex and let $\sigma$ be a face of dimension $d-1$ in $\Delta$. Then
$$\reg(\Delta) \le \max \{\reg(\link_\Delta(\sigma)) + d, \reg(\del_\Delta(\sigma))\}.$$
\item Let $H$ be a simple hypergraph and let $V$ be a collection of $d$ vertices in $H$. Then
$$\reg(H) \le \max\{ \reg(H : V) + d, \reg(H + V)\}.$$
\end{enumerate}
\end{theorem}

Another basic exact sequence in commutative algebra is:
$$0 \longrightarrow \frac{R}{I \cap J} \longrightarrow \frac{R}{I} \oplus \frac{R}{J} \longrightarrow \frac{R}{I+J} \longrightarrow 0.$$
Let $E$ be an edge of a simple hypergraph $H$. By taking $I = I(H \setminus E)$ and $J = (x^E)$, this sequence gives
$$0 \longrightarrow \frac{R}{(x^E) \cap I(H \setminus E)} \longrightarrow \frac{R}{(x^E)} \oplus \frac{R}{I(H \setminus E)} \longrightarrow \frac{R}{I(H)} \longrightarrow 0.$$
Taking the associated long exact sequence of cohomology modules again, we get
\begin{align}
\reg(H) \le \max\{|E|, \reg(H \setminus E), \reg\big((x^E) \cap I(H \setminus E)\big) - 1\}. \label{eq.splitedge}
\end{align}

Recall that $H_E$ is the contraction of $H \setminus N(E)$ to the vertices $X \setminus N[E]$. It is easy to see that $(x^E) \cap I(H \setminus E) = x^E(y ~|~ y \in N(E)) + I(H_E)$. Also, since the variables in $N[E]$ do not appear in $H_E$, by taking the tensor product of minimal free resolutions, we get
$$\reg\big((x^E) \cap I(H \setminus E)\big) = \reg(I(H_E)) + |E|.$$
Thus, (\ref{eq.splitedge}) gives the following inductive bound.

\begin{theorem} \label{thm.ind3}
Let $H$ be a simple hypergraph and let $E$ be an edge of cardinality $d$ in $H$. Then
$$\reg(H) \le \max\{d, \reg(H \setminus E), \reg(H_E) + d-1\}.$$
\end{theorem}

Induction also works even when we do not necessarily split the edges of $H$ into disjoint subsets. Kalai and Meshulam \cite{KM} obtain the following powerful result. This result was later extended to arbitrary (not necessarily squarefree) monomial ideals by Herzog \cite{Herzog}.

\begin{theorem} \label{thm.ind2}
Let $I_1, \dots, I_s$ be squarefree monomial ideals in $R$. Then
$$\reg\left(R\Big/\sum_{i=1}^s I_i\right) \le \sum_{i=1}^s \reg(R/I_i).$$
\end{theorem}

In particular, for edge ideals of a hypergraph and subhypergraphs, we have the following inductive bound.

\begin{corollary} \label{cor.ind2}
Let $H$ and $H_1, \dots, H_s$ be simple hypergraphs over the same vertex set $X$ such that $\E(H) = \bigcup_{i=1}^s \E(H_i)$. Then
$$\reg(R/I(H)) \le \sum_{i=1}^s \reg(R/I(H_i)).$$
\end{corollary}

%%%%%%%%%%%%%%%%%%%%%%%%%%%%%%%%%%%%%%%%%%%%%%%%%%%%%%%%%%%%%%%%%%%%%%%%%%%%%%

\section{Combinatorial bounds for regularity} \label{sec.bound}

In this section, we examine various bounds for the regularity of a squarefree monomial ideal in terms of combinatorial data from associated simplicial complex and hypergraph. These bounds can be proved using inductive results from Section \ref{sec.induction}. There are also works in the literature that relate the regularity of (squarefree) monomial ideals to other algebraic invariants (cf. \cite{AA, C, FKT, HT}). These works are beyond the scope of this survey. %There are also studies on the regularity of squarefree monomial ideals via other combinatorial constructions, for example, \emph{path ideals} of directed trees (cf. \cite{BHO}). We shall not include constructions that do not necessarily give a one-to-one correspondence between squarefree monomial ideals and their combinatorial counterparts.

For edge ideals of graphs, it turns out that the matching number and induced matching number provide nice upper and lower bounds for the regularity. The following result is due to Katzman \cite[Lemma 2.2]{Katzman}.

\begin{theorem} \label{thm.inducedmatchinggraph}
Let $G$ be a simple graph and let $\nu(G)$ be the maximum size of an induced matching in $G$. Then
$$\reg(I(G)) \ge \nu(G)+1.$$
\end{theorem}

This result is generalized in \cite[Theorem 6.5]{HVT2008} for properly connected hypergraph, and extended for all simple hypergraph in \cite[Corollary 3.9]{MV}. In fact, the result is a direct consequence of the inductive bound in Lemma \ref{lem:RegShrinksInSubhypergraphs}.

\begin{theorem} \label{thm.inducedmatching}
Let $H$ be a simple hypergraph. Suppose that $\{E_1, \dots, E_s\}$ forms an induced matching in $H$. Then
$$\reg(H) \ge \sum_{i=1}^s (|E_i|-1) + 1.$$
\end{theorem}

\begin{proof} Let $H'$ be the induced subhypergraph of $H$ on the vertex set $\bigcup_{i=1}^s E_i$. Since $\{E_1, \dots, E_s\}$ forms an induced matching in $H$, these are the only edges in $H'$. Thus,
$$\reg(H') = \sum_{i=1}^s (|E_i|-1) + 1.$$
Moreover, by Lemma \ref{lem:RegShrinksInSubhypergraphs}, $\reg(H) \ge \reg(H')$. The result now follows.
\end{proof}

\begin{remark} It is clear that if $H$ consists of disjoint edges then the bound in Theorem \ref{thm.inducedmatching} becomes an equality. On the other hand, as we shall see later in Example \ref{ex.matching}, the regularity of a hypergraph can be arbitrarily larger than the right hand side of the bound in Theorem \ref{thm.inducedmatching}.
\end{remark}

Lower bounds and upper bounds are more interesting when they go together. A natural invariant related to the induced matching number is the (\emph{minimax}) matching number. The following result was proved in \cite[Theorem 6.7]{HVT2008} and \cite[Theorem 11]{Russ}.

\begin{theorem} \label{thm.matching}
Let $G$ be a simple graph. Let $\beta(G)$ be the minimum size of a maximal matching $G$. Then
$$\reg(G) \le \beta(G) + 1.$$
\end{theorem}

We can also recover the proof of Theorem \ref{thm.matching} as a direct consequence of an inductive bound, Corollary \ref{cor.ind2}. For that we first need a simple lemma.

\begin{lemma} \label{lem.reg1}
Let $G$ be a simple graph and assume that an edge $\{u,v\}$ forms a maximal matching of size 1 in $G$. Then
$$\reg(R/I(G)) = 1.$$
\end{lemma}

\begin{proof} Since $\reg(R/I(G)) = \reg(I(G))-1$, it suffices to show that $\reg(I(G)) = 2$. The assertion is trivial if $G$ consists of exactly one edge $\{u,v\}$. We can then induct on the number of edges in $G$, making use of Theorem \ref{thm.ind11}.
\end{proof}

\begin{proof}[of Theorem \ref{thm.matching}]
Let $\beta = \beta(G)$ and let $\{E_1, \dots, E_\beta\}$ be a maximal matching in $G$. For each $E_i$, let $G_i$ be the subgraph of $G$ consisting of $E_i$ and all edges incident to its vertices. Since $\{E_1, \dots, E_\beta\}$ forms a maximal matching in $G$, we have $\bigcup_{i=1}^\beta \E(G_i) = \E(G)$. The result now follows from Corollary \ref{cor.ind2} and Lemma \ref{lem.reg1}.
\end{proof}

\begin{remark} If $G$ consists of disjoint edges, then the bound in Theorem \ref{thm.matching} becomes an equality. On the other hand, by taking $G$ to be the complement of a chordal graph (in which case, $\reg(I(G)) = 2$), one can make $\beta(G)+1$ arbitrarily larger than $\reg(I(G))$.
\end{remark}

A number of generalizations for Theorem \ref{thm.matching} has been developed. Woodroofe \cite{Russ} was the first to observe that giving a matching in a graph $G$ is a simple way to ``cover'' the edges of $G$ by subgraphs. One can also consider giving a matching in $G$ as a special way to ``pack'' edges in $G$. These ideas have been extended to give better bounds for graphs and to obtain bounds for hypergraphs in general.

Let $c(G)$ denote the minimum number of subgraphs in $G$ whose complements are chordal such that every edge in $G$ belong to at least one of those subgraphs. Woodroofe \cite[Lemma 1]{Russ} extends Theorem \ref{thm.matching} to give the following bound.

\begin{theorem} \label{thm.Woodroofe}
Let $G$ be a simple graph. Then
$$\reg(G) \le c(G) + 1.$$
\end{theorem}

Moradi and Kiani \cite[Theorem 2.1]{MK} improve the bound in Theorem \ref{thm.matching} for the class of vertex decomposable graphs. Define $\gamma(G)$ to be the maximum number of vertex disjoint paths of length at most 2 in $G$ such that paths of length 1 form an induced matching.

\begin{theorem} \label{thm.MK}
Let $G$ be a vertex decomposable graph. Then
$$\reg(G) \le \gamma(G) + 1.$$
\end{theorem}

The author and Woodroofe \cite{HW} extend this result to any graph making use of packing of stars in a graph. A \emph{star} consists a vertex $x$ as its \emph{center} and edges $xy_i$s incident to $x$ (the vertices $y_i$s are referred to as the \emph{outer} vertices of the star). A star is \emph{nondegenerate} if there are at least 2 edges incident to its center. A collection $P$ of stars in $G$ is called \emph{center-separated} if for any pair of stars $S_1$ and $S_2$ in $P$, at least 2 outer vertices and the center of $S_1$ are not contained in $S_2$ (and vice-versa). For a center-separated collection of nondegenerate stars $P$, let $\zeta_P$ denote the number of stars in $P$ plus the size of the remaining induced matching after deleting the vertices of $P$ from $G$. Let $\zeta(G)$ be the maximum value of $\zeta_P$ over all center-separated collection of nondegenerate stars $P$. It is not hard to see that $\zeta(G) \le \gamma(G)$.

\begin{theorem} \label{thm.starpacking}
Let $G$ be a simple graph. Then
$$\reg(G) \le \zeta(G) + 1.$$
\end{theorem}

\begin{example} Let $G$ be the graph in Figure \ref{fig:Graph}. It is easy to see that $\zeta(G) = 1$. For instance, the nondegenerate star centered at $b$ gives a maximal center-separated packing of stars in $G$, whose removal results in a subgraph of isolated vertices $\{a,c\}$. The bound in Theorem \ref{thm.starpacking} gives $\reg(I(G)) \le 2$. In fact, we have the equality for this example.
\end{example}

For hypergraphs the problem becomes more subtle. In particular, using the matching number to bound the regularity as in Theorem \ref{thm.matching} is no longer possible.

\begin{figure}[hbtf]
\centering
\includegraphics[height=1.5in]{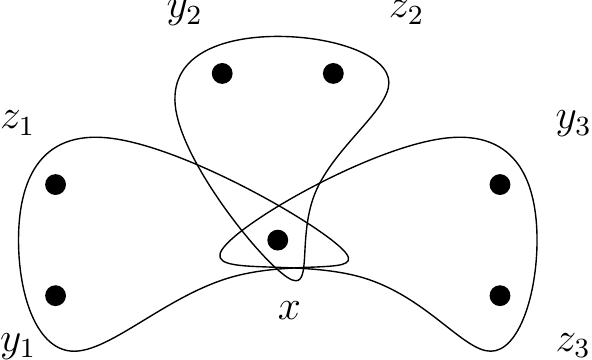}
\caption{A 3-uniform hypergraph with regularity greater than $(3-1) \cdot \beta$.} \label{fig:MatchingBdFailure}
\end{figure}

\begin{example} \label{ex.matching}
For $s>1$, consider the hypergraph $H_{s}$ with
edges $\{x,y_{i},z_{i}\}$ (for $i=1,\dots,s$)
on the vertex set $\{x,y_{1},\dots,y_{s},z_{1},\dots,z_{s}\}$. Figure~\ref{fig:MatchingBdFailure}
illustrates $H_{3}$. We have that the induced matching number and minimax
matching number of $H_{s}$ are both 1. On the other hand, it is
straightforward to compute that $\reg(I(H_{s}))=s+2$, which can
be taken to be arbitrarily far from $(3-1)\beta(H_{s})= 2$.
\end{example}

A generalization of Theorem \ref{thm.matching} is obtained in \cite[Theorem 1.2]{HW} replacing the notion of a matching by a 2-collage.

\begin{definition} Let $H$ be a simple hypergraph. A subset $C$ of the edges in $H$ is called a \emph{2-collage} if for each edge $E$ in $H$, there exists a vertex $v \in E$ such that $E \setminus \{v\}$ is contained in some edge of $C$. For a uniform hypergraph, this condition is equivalent to requiring that for any edge $E$ in $H$ there is an edge $F$ in $C$ such that the symmetric difference of $E$ and $F$ has cardinality exactly 2.
\end{definition}

\begin{theorem} \label{thm.2collage}
Let $H$ be a simple hypergraph, and let $\{E_1, \dots, E_s\}$ be a 2-collage in $H$. Then
$$\reg(H) \le \sum_{i=1}^s (|E_i|-1) + 1.$$
\end{theorem}

Instead of packing a 2-collage in a hypergraph, the concept of \emph{edgewise domination} was introduced by Dao and Schweig \cite{DS}, as a covering invariant, to generalize Theorem \ref{thm.matching}.

\begin{definition}
A collection $C$ of edges in a hypergraph $H$ is called \emph{edgewise dominant} if every vertex $x \in H^{\textrm{red}}$ not contained in some edge of $C$ or contained in an isolated loop has a neighbor contained in some edge of $C$. Define
$$\epsilon(H) = \min \{ |C| ~\big|~ C \textrm{ is edgewise dominant}\}.$$
\end{definition}

\begin{theorem} \label{thm.edgedominant}
For any simple hypergraph $H$, we have
$$\reg(H) \le |X(H)| - \epsilon(H^\vee).$$
\end{theorem}

To prove Theorem \ref{thm.edgedominant}, the inductive method of Theorem \ref{thm.ind11} was put in a more general perspective by considering a hereditary family of hypergraphs that allows one to go from a given hypergraph $H$ to the deletion and contraction of $H$ at a vertex.

\begin{definition} Let $\Phi$ be a collection of simple hypergraphs. We say that $\Phi$ is \emph{hereditary} if for any hypergraph $H \in \Phi$ and any subset $V$ of the vertices of $H$, $H^{\textrm{red}}$, $H + V$ and $H : V$ are all in $\Phi$.
\end{definition}

The proof of Theorem \ref{thm.edgedominant} is based on the following more general version of Theorem \ref{thm.ind11} (see \cite[Lemma 3.3]{DS}).

\begin{theorem} \label{thm.hereditary}
Let $\Phi$ be a hereditary family of simple hypergraphs, and let $f: \Phi \rightarrow \NN$ be a function such that $f(H^{\textrm{red}}) = f(H)$ for all $H \in \Phi$, $f(H) \le |X(H)|$ when $H$ contains no edges, and $f(H) = 0$ when $H$ has only isolated loops. Furthermore, suppose that $f$ satisfies the following condition: for any $H \in \Phi$ with at least one edge of cardinality $\ge 2$, there exists a sequence of subsets $A_1, \dots, A_t$ of the vertices such that, writing $H_i$ for the hypergraph corresponding to $H + \sum_{j=1}^i A_j$, the following two properties are satisfied.
\begin{itemize}
\item $|\textrm{is}(H_t)| > 0$ and $f(H_t) + |\textrm{is}(H_t)| \ge f(H)$, and
\item for each $i$, $f(H_{i-1} : A_i) + |\textrm{is}(H_{i-1})| + |A_i| \ge f(C)$.
\end{itemize}
Then for any $H \in \Phi$, we have
$$\reg(H^\vee) = \pd(H) \le |X(H^{\textrm{red}})| - f(H).$$
\end{theorem}

This technique, restricted to graphs \cite{DHS}, gives the following numerical bound for the regularity (see \cite[Theorem 4.1]{DHS}).

\begin{theorem} \label{thm.log}
Let $G$ be a graph. Assume that the resolution of $I(G)$ is linear up to the $k$-th step for some $k \ge 1$, and let $d$ be the maximum degree of a vertex in $G$. Then
$$\reg(G) \le \log_{\frac{k+4}{2}}\left(\dfrac{d}{k+1}\right)+3.$$
\end{theorem}

\begin{remark} It is not clear how strong the bound in Theorem \ref{thm.log} is in comparison to other known bounds using packing and covering invariants. %There does not seem to be many examples showing how sharp this bound is either.
\end{remark}

Lin and McCullough, in \cite{LM}, introduced the notion of \emph{labeled hypergraphs} that, in complement to the well-studied constructions of Stanley-Reisner ideals and edge ideals, also gives a one-to-one correspondence to squarefree monomial ideals, and used this notion to study the regularity of squarefree monomial ideals. Note that the \emph{unlabeled} version of their construction coincides with the \emph{dual hypergraph} notion in hypergraph theory (cf. \cite{Berge}). In the same spirit of employing induction, Lin and McCullough \cite[Theorem 4.9]{LM} obtain the following bound. We shall rephrase their result in terms of edge ideals of hypergraphs.

An edge $E$ of a simple hypergraph $H$ is said to contain a \emph{free vertex} if there exists $x \in E$ such that $x$ does not belong to any other edge in $H$.

\begin{theorem} \label{thm.labelhypergraph}
Let $H = (X,\E)$ be a simple hypergraph, and let $H' = (X, \E')$ be the hypergraph obtained by removing all edges with free vertices from $H$. Let $\beta(H')$ be the matching number of $H'$. Then
$$\reg(H) \le |X| - |\E| + |\E'| - \beta(H') + 1.$$
\end{theorem}

\begin{figure}[h!]
\centering
\includegraphics[height=1.5in]{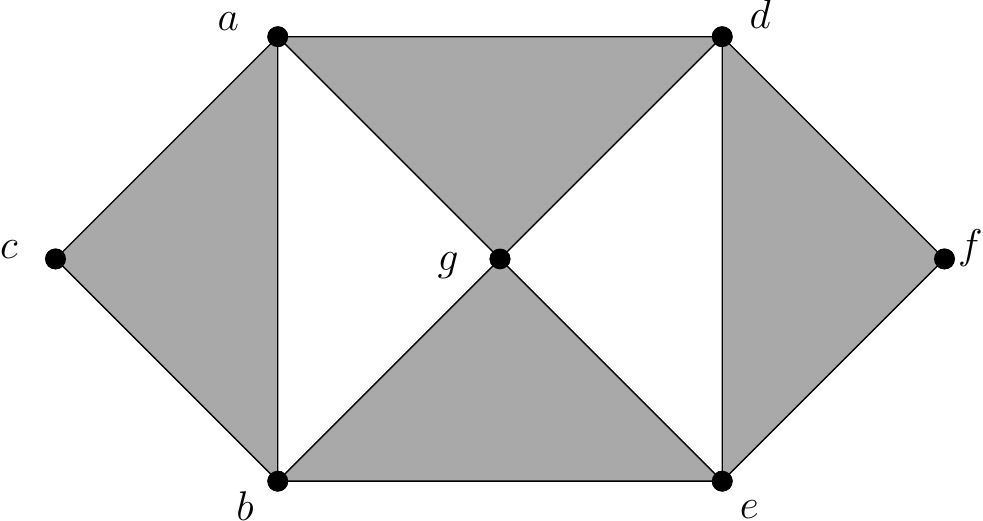}
\caption{Comparing bounds in Theorems \ref{thm.2collage}, \ref{thm.edgedominant} and \ref{thm.labelhypergraph}.} \label{fig:Example2}
\end{figure}

\begin{remark}
It is not quite clear which bounds among Theorems \ref{thm.2collage}, \ref{thm.edgedominant} and \ref{thm.labelhypergraph} are best in general. The following examples were given in \cite{LM} to illustrate this.
\end{remark}

\begin{example} Let $I = (abc, def, adg, beg)$ be the edge ideal of the hypergraph in Figure \ref{fig:Example2}. Then the bounds for $\reg(I)$ in Theorems \ref{thm.2collage} and \ref{thm.edgedominant}, respectively, are 9 and 6. On the other hand, the bound for $\reg(I)$ in Theorem \ref{thm.labelhypergraph} is 5.
\end{example}

\begin{figure}[h!]
\centering
\includegraphics[height=1.5in]{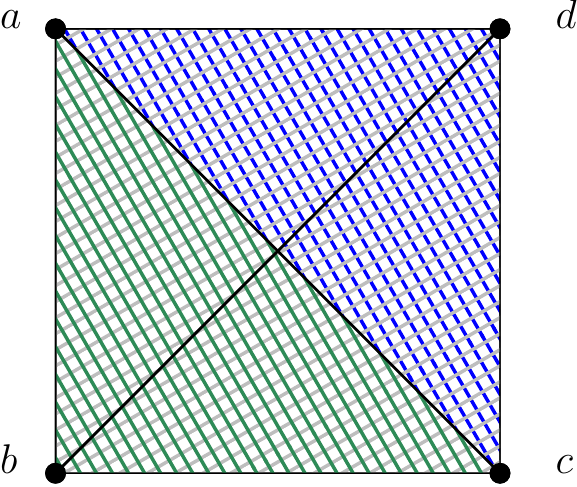}
\caption{Comparing bounds in Theorems \ref{thm.2collage}, \ref{thm.edgedominant} and \ref{thm.labelhypergraph}.} \label{fig:Example3}
\end{figure}

\begin{example} Let $I = (abc, abd, acd, bcd)$ be the edge ideal of the hypergraph in Figure \ref{fig:Example3}. Then the bounds for $\reg(I)$ in Theorems \ref{thm.2collage} and \ref{thm.edgedominant} are both 3, while the bound using labeled hypergraph in Theorem \ref{thm.labelhypergraph} is 4.
\end{example}

%%%%%%%%%%%%%%%%%%%%%%%%%%%%%%%%%%%%%%%%%%%%%%%%%%%%%%%%%%%%%%%%%%%%%%%%%%%%%%

\section{Small regularity and computing regularity} \label{sec.compute}

In this section we survey studies that explicitly compute the regularity for special classes of squarefree monomial ideals, and identify combinatorial structures that force the ideals to have small regularity. Ideals of regularity 0 are trivial. A squarefree monomial ideal has regularity 1 if and only if it is generated by a collection of variables. Thus, we shall only consider ideals with regularity at least 2.

The following result was originally stated and proved by Wegner \cite{Wegner} using topological language, and re-stated in terms of monomial ideals by Fr\"oberg \cite{Fr} (see also \cite{EGHP, HVT2005}). It is in fact a direct consequence of Lemma \ref{lem:RegFromTopology}.

\begin{theorem} \label{thm.Froberg}
Let $G$ be a simple graph. Then $\reg(I(G)) = 2$ if and only if $G^c$ is a chordal graph.
\end{theorem}

\begin{proof} We briefly sketch out the proof. Let $\Delta$ be the independence complex of $G$. By Lemma \ref{lem:RegFromTopology}, $\reg(R/I(G)) = 1$ iff no induced subcomplex of $\Delta$ has any balls of positive dimension. This is the case iff the 1-skeleton of $\Delta$ has no induced cycle of length $\ge 4$. Note that the 1-skeleton of $\Delta$ is exactly the complement of $G$. In fact, an induced cycle of length $\ge 4$ in $G^c$ gives a homology cycle in $\Delta$.
\end{proof}

It is then natural to investigate squarefree monomial ideals with regularity 3. To the best of our knowledge, it is still an open question to characterize these ideals; though partial results have been obtained.

Let $P$ be a given collection of graphs. A graph $G$ is said to be \emph{$P$-free} if it contains no induced subgraphs that are the same as elements in $P$. In particular, a graph $G$ is called \emph{claw-free} if it contains no 4 vertices on which the induced subgraph is a star, and a graph $G$ is \emph{$C_4$-free} if it contains no induced $4$-cycles. The following result was proved in \cite[Theorem 1.2]{Nevo} (see also \cite[Theorem 3.4]{DHS}).

\begin{theorem} \label{thm.clawfree}
Let $G$ be a claw-free simple graph such that $G^c$ contains no $C_4$. Then
$$\reg(I(G)) \le 3.$$
\end{theorem}

\begin{example} \label{ex.C4}
A 5-cycle is a claw-free graph whose complement contains no $C_4$, and $\reg(I(C_5)) = 3$. %This shows that the class of graphs with regularity $\le 3$ is strictly larger than that of claw-free graphs whose complement contains no $C_4$.
\end{example}

For bipartite graphs, the edge ideals having regularity 3 can be characterized. Recall that a graph $G$ is \emph{bipartite} if the vertices $X(G)$ can be partitioned into disjoint subsets $X(G) = Y \cup Z$ such that edges of $G$ connect a vertex in $Y$ and a vertex in $Z$. The \emph{bipartite complement} of a bipartite graph $G$, denoted by $G^{\textrm{bc}}$, is the bipartite graph over the same partition of the vertices, and for $y \in Y, z \in Z$, $\{y,z\} \in G^{\textrm{bc}}$ if and only if $\{y,z\} \not\in G$. The following classification was given in \cite[Theorem 3.1]{FRG}.

\begin{theorem} \label{thm.FRG}
Let $G$ be a connected bipartite graph. Then $\reg(G) = 3$ if and only if $G^c$ has induced cycles (of length $\ge 4$) and $G^{\textrm{bc}}$ has no induced cycles of length $\ge 6$.
\end{theorem}

Let us now turn to classes of ideals for which the regularity can be computed explicitly. In most of the studies that we survey, this is the case when inductive inequalities in Section \ref{sec.induction} become equalities, and the method is to find combinatorial invariants that behaves well with respect to the induction processes. %It also happens that in known cases when the regularity can be computed explicitly, it is given by the induced matching number.%To this end, it is important to know when the inequalities in Theorems \ref{thm.ind11}, \ref{thm.ind11} and \ref{thm.ind3}, or more generally in (\ref{eq.simplereg}) and (\ref{eq.splitedge}), become equalities.

The following result was first proved for forests in \cite[Theorem 2.18]{Z} and \cite[Corollary 3.11]{HVT2005}, and extended to chordal graphs in \cite[Theorem 6.8]{HVT2008}.

\begin{theorem} \label{thm.regchordal}
Let $G$ be a chordal graph and let $\nu(G)$ be the maximum size of an induced matching in $G$. Then
$$\reg(G) = \nu(G)+1.$$
\end{theorem}

To prove Theorem \ref{thm.regchordal}, we shall need a lemma (see \cite[Lemma 5.7 and Theorem 6.2]{HVT2008}) which states that for chordal graphs, one can always find an edge $E$ such that the inductive inequality of Theorem \ref{thm.ind3} becomes an equality, and that this process can be done recursively.

\begin{lemma} \label{lem.splittingedge}
Let $G$ be a chordal graph. Then there exists an edge $E$ in $G$ such that
\begin{enumerate}
\item $\reg(G) = \max\{\reg(G_E) + 1, \reg(G \setminus E)\}.$
\item $G \setminus E$ and $G_E$ are both chordal graphs.
\end{enumerate}
\end{lemma}

\begin{proof}[of Theorem \ref{thm.regchordal}]
We use induction on the number of edges in $G$. If $G$ consists of only isolated vertices or a single edge then the assertion is trivial. For the induction step, in light of Lemma \ref{lem.splittingedge} it suffices to show that for an edge $E$ in $G$ satisfying the conclusion of Lemma \ref{lem.splittingedge}, we have
$$\nu(G) = \max\{\nu(G_E)+1, \nu(G\setminus E)\}.$$
This is indeed true. It is easy to see that an induced matching in $G \setminus E$ is also an induced matching in $G$, and an induced matching in $G_E$ together with $E$ would form an induced matching in $G$. Thus, $\nu(G) \ge \max\{\nu(G_E)+1, \nu(G\setminus E)\}.$ On the other hand, suppose that $C = \{E_1, \dots, E_c\}$ is an induced matching in $G$ of maximum size. If $E \not\in C$ then $C$ is an induced matching in $G \setminus E$. If $E \in C$, then $C \setminus \{E\}$ forms an induced matching in $G_E$. Therefore, $\nu(G) \le \max\{\nu(G_E)+1, \nu(G\setminus E)\}.$
\end{proof}

\begin{figure}[h!]
\centering
\includegraphics[height=1.5in]{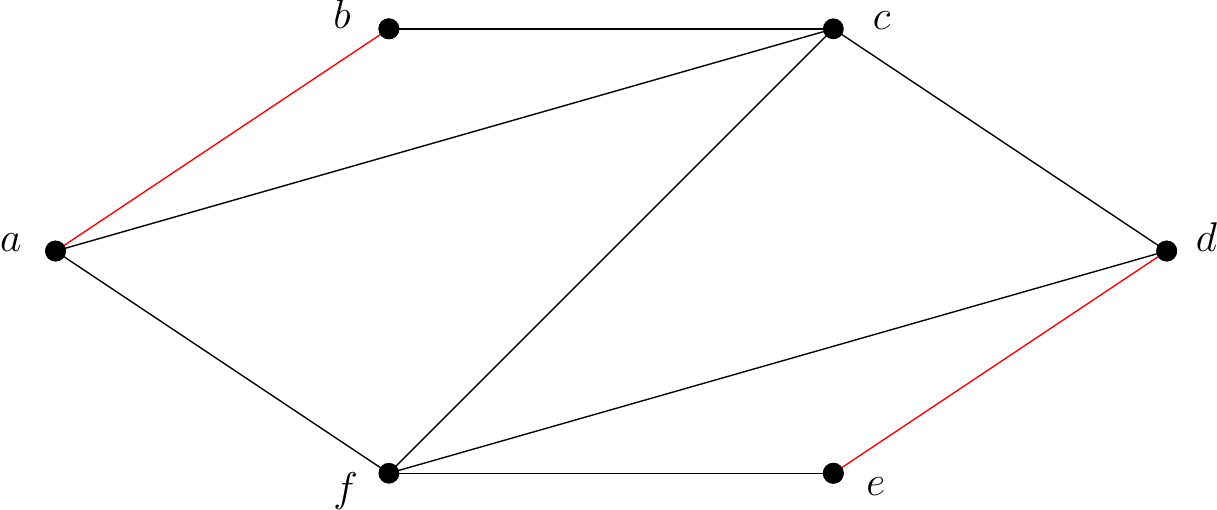}
\caption{A chordal graph.}\label{fig:chordal}
\end{figure}

\begin{example} Let $G$ be the chordal graph in Figure \ref{fig:chordal} and let $I = I(G)$. Then
$$0 \rightarrow R(-5)^2 \oplus R(-6) \rightarrow R(-4)^{10} \oplus R(-5)^2 \rightarrow R(-3)^{16} \oplus R(-4) \rightarrow R(-2)^9 \rightarrow I \rightarrow 0$$
is the minimal free resolution of $I$ giving $\reg(I) = 3$. On the other hand, it is easy to see that $\nu(G) = 2$ and $\{ab, cd\}$ is a maximal induced matching.
\end{example}

Recall that a collection of vertices $V$ in a graph $G$ is called a \emph{vertex cover} if for any edge $E$ in $G$, $V \cap E \not= \emptyset$. A \emph{minimal} vertex cover is with respect to inclusion. Note that the smallest size of a vertex cover in $G$ is equal to the height of $I(G)$. A graph $G$ is called \emph{unmixed} if all minimal vertex cover of $G$ has the same cardinality. For an unmixed graph $G$, it is known (cf. \cite{GV}) that $2\height I(G) \ge |X(G)|$. A graph $G$ is called \emph{very well-covered} (see \cite{many}) if $G$ is unmixed, has no isolated vertices, and $2\height I(G) = |X(G)|$.

Theorem \ref{thm.regchordal} has been extended to a number of classes of graphs. In particular, for the following classes of graph, the regularity can be computed by its induced matching number.

\begin{theorem} \label{thm.computeimn}
Let $G$ be a simple graph. Let $\nu(G)$ be the maximum size of an induced matching in $G$. Then
$$\reg(G) = \nu(G)+1$$
in the following cases:
\begin{enumerate}
%\item $G$ is a chordal graph (\cite{HVT2005, Z});
\item $G$ is a sequentially Cohen-Macaulay bipartite graph (see \cite{VT2009});
\item $G$ is an unmixed bipartite graph (see \cite{Kummini});
\item $G$ is a very well-covered graph (see \cite{many});
%\item $G$ is a $(C_4,C_5)$-free vertex decomposable graph (see \cite{BC2012});
\item $G$ is a $C_5$-free vertex decomposable graph (see \cite{KAM}, the case where $G$ is also $C_4$-free was proved in \cite{BC2012}).
\end{enumerate}
\end{theorem}

\begin{remark} Chordal and sequentially Cohen-Macaulay bipartite graphs are vertex decomposable. Also, bipartite graphs are $C_5$-free. Therefore, in Theorem \ref{thm.computeimn}, 1. follows from 4.
\end{remark}

\begin{example} \label{ex.C5}
Let $G$ be the induced $5$-cycle. Then $G$ is not $C_5$-free nor very well-covered. Its independence complex is also an induced $5$-cycle, which is vertex decomposable. Thus, the class of vertex decomposable graphs is not fully covered by Theorem \ref{thm.computeimn}.

Moreover, the minimal free resolution of $I = I(G)$ is given by
$$0 \rightarrow R(-5) \rightarrow R(-3)^5 \rightarrow R(-2)^5 \rightarrow I \rightarrow 0.$$
Therefore, $\reg(I) = 3$. On the other hand, the induced matching number of $G$ is $\nu(G) = 1$. Observe that $\reg(I(G)) \not= \nu(G) + 1$ in this example.
\end{example}

The problem is again more subtle when moving to hypergraphs. To do so, we need to identify classes of hypergraphs where inductive inequalities in Section \ref{sec.induction} become equalities, and to find combinatorial invariants of hypergraphs that respect induction processes.

It is proved by the author and Woodroofe in \cite[Theorem 1.5]{HW} that the inductive inequalities of Theorem \ref{thm.ind11} are equalities for all vertex decomposable simplicial complexes and hypergraphs.

\begin{theorem} \label{thm.vd}
Let $H$ be a simple hypergraph with edge ideal $I = I(H)$. Suppose that $H$ is vertex decomposable and $x$ is the initial vertex in its shedding order. Then
$$\reg(I) = \max\{\reg(I:x) + 1, \reg(I,x)\}.$$
\end{theorem}

The following example arises from communication with Chris Francisco. The author would like to thank Chris Francisco for his help.

\begin{example} \label{ex.smaller}
Let $I = (acd, bcd, abe, bce, bcf, cdf) \subseteq R = K[a,\dots, f]$. Then $\reg(I) = 3$, $\reg(I,a) = 3 = \reg(I)$ and $\reg(I:a)+1 = 4 > \reg(I)$. Thus, in this example, $\reg(I)$ is equal to the smaller value between $\reg(I,a)$ and $\reg(I:a)+1$.
\end{example}

Example \ref{ex.C5} shows that even for vertex decomposable graphs, the induced matching number is no longer the right invariant to compute the regularity (i.e., the regularity may be strictly bigger than the induced matching number). On the other hand, for a smaller class of vertex decomposable hypergraphs, those that are triangulated, it was proved in \cite{HVT2008} that this invariant still works.

\emph{Triangulated} hypergraphs were introduced in \cite{HVT2008} as a generalization of chordal graphs. Recall that the \emph{ distance} between 2 edges $E$ and $F$ in a $d$-uniform hypergraph $H$ is the minimum length of a sequence of edges $E_0 = E, E_1, \dots, E_s = F$ such that $|E_{i-1} \cap E_i| = d-1$. A $d$-uniform hypergraph $H$ is called \emph{properly connected} if for any intersecting edges $E$ and $F$, the distance between $E$ and $F$ is exactly $d- |E \cap F|$. Note that all simple graph are properly connected. Recall also that for a vertex $x$, the set of \emph{neighbors} of $x$ is given by
$$N(x) = \{y \in X ~|~ \exists E \in \E \textrm{ s.t. } \{x,y\} \subseteq E\}.$$

\begin{definition}
A $d$-uniform properly connected hypergraph $H = (X,\E)$ is said to be \emph{triangulated} if for every nonempty subset $Y \subseteq X$, the induced subgraph $H[Y]$ contains a vertex $x \in Y \subseteq X$ such that the induced hypergraph of $H[Y]$ on $N(x)$ is a complete hypergraph consisting of all subsets of $d$ elements of $N(x)$.
\end{definition}

Note that triangulated hypergraphs are vertex decomposable. The following theorem was proved in \cite[Theorem 6.8]{HVT2008}.

\begin{theorem} \label{thm.triangulated}
Let $H$ be a $d$-uniform properly connected hypergraph, and assume that $H$ is triangulated. Let $\nu(H)$ denote the induced matching number of $H$. Then
$$\reg(H) = (d-1)\nu(H) + 1.$$
\end{theorem}

Using their notion of labeled hypergraphs, Lin and McCullough \cite[Proposition 4.1 and Theorem 4.12]{LM} also compute the regularity explicitly for special classes of hypergraphs. We shall again rephrase their results in terms of edge ideals of hypergraphs.
A hypergraph $H$ is called \emph{saturated} if every edge in $H$ contains free vertices.

\begin{theorem} \label{thm.LMsaturated}
Let $H = (X, \E)$ be a simple hypergraph with edge ideal $I = I(H)$. The following are equivalent:
\begin{enumerate}
\item $H$ is saturated; and
\item The Taylor resolution of $I$ is minimal.
\end{enumerate}
In this case, we have
$$\reg(I) = |X| - |\E| + 1.$$
\end{theorem}

For a vertex $x$ in a hypergraph $H = (X, \E)$, let $W_x = \{E \in \E ~|~ x \in E\}$. Denote by $S_E$ the set $\{W_x ~|~ x \in E\}$, and let $\P = \bigcup_{E \in \E} S_E$. An element $W_x$ in $\P$ is called \emph{simple} if it does not contain any other element of $\P$ as a proper subset. The hypergraph $H$ is said to have \emph{isolated simple edges} if for any edge $E$ in $H$ not containing free vertices, the set $S_E$ has exactly one simple maximal element.

\begin{theorem} \label{thm.LMsimpleedges}
Let $H$ be a simple hypergraph with edge ideal $I = I(H)$. Suppose that $H$ has isolated simple edges. Then
$$\reg(I) = |X| - |V| + \sum_{W \in \P, \ W \textrm{ simple}} (|W|-1) + 1.$$
\end{theorem}

\begin{figure}[h!]
\centering
\includegraphics[height=1.5in]{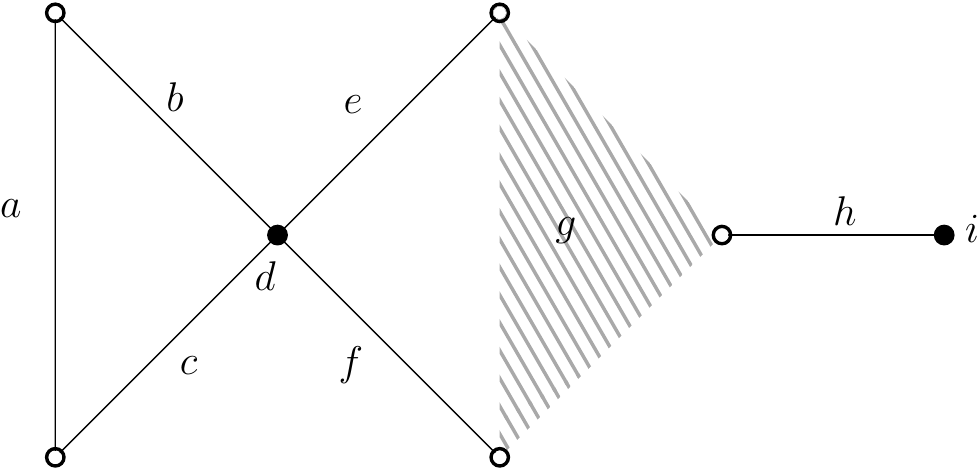}
\caption{A labeled hypergraph having isolated simple edges.}\label{fig:ExampleLH}
\end{figure}

\begin{example} \label{ex.ExampleLH}
Let $I = (ab, bcdef, ac, eg, fg, gh, hi) \subseteq R = K[a,b,\dots, i]$ be the edge ideal of $H$. Then the set $\P$ of $H$ is depicted in Figure \ref{fig:ExampleLH}, where filled points represent edges in $H$ with free vertices and unfilled points represent edges in $H$ without free vertices. Note that, as a hypergraph, this figure represents the dual hypergraph of $H$, and also the labeled hypergraph corresponding to $I$ in the sense of \cite{LM}. It can be seen that $H$ has isolated simple edges, and those simple sets $W_x$s in $\P$ are labeled by $a$ and $g$. Theorem \ref{thm.LMsimpleedges} gives
$$\reg(I) = 9 - 7 + (2-1) + (3-1) + 1 = 6.$$
\end{example}

\begin{remark} It can be seen that the ideal $I$ in Example \ref{ex.ExampleLH} is not sequentially Cohen-Macaulay. Thus, its associated simplicial complex $\Delta(I)$ is not vertex decomposable. This example shows that inductive inequalities in Theorems \ref{thm.ind11} and \ref{thm.vd} may become equalities for a larger class of squarefree monomial ideals than those corresponding to vertex decomposable simplicial complexes. %It would be interesting to compare the class of squarefree monomial ideals $I$ with vertex decomposable simplicial complexes $\Delta(I)$ and those with labeled hypergraphs $\H(I)$ having isolated simple edges.
\end{remark}

%%%%%%%%%%%%%%%%%%%%%%%%%%%%%%%%%%%%%%%%%%%%%%%%%%%%%%%%%%%%%%%%%%%%%%%%%%%%%%

\section{Open problems and questions} \label{sec.prob}

In this last section of the paper, we state a number of open problems and conjectures that we would like to see answered. We first observe that the combinatorial bounds in Theorems \ref{thm.2collage}, \ref{thm.edgedominant} and \ref{thm.labelhypergraph} in practice are not easy to compute. It is desirable to find bounds that relate to more familiar combinatorial invariants of hypergraphs and simplicial complexes.

\begin{problem} \label{prob.bound}
Find bounds for the regularity of a squarefree monomial ideal in terms of familiar combinatorial invariants and structures of associated hypergraph and simplicial complex.
\end{problem}

Let $\chi(H)$ denote the \emph{chromatic number} of a hypergraph $H$, i.e., the least number of colors needed to color the vertices of $H$ such that no edge (that is not an isolated loop) in $H$ is mono-colored. It was observed in \cite[Theorem 2]{Russ} that for any independent set $T$ in $G$,
\begin{align}
\reg(R/I(G)) \le \chi((G \setminus T)^c). \label{eq.chromatic}
\end{align}

\begin{problem} \label{prob.chromatic}
Extend (\ref{eq.chromatic}) to hypergraphs. Investigate the relation between the regularity of a hypergraph and its coloring properties.
\end{problem}

As pointed out in Example \ref{ex.C4}, the class of graphs with regularity $\le 3$ is strictly larger than that of claw-free graphs whose complements contain no induced $C_4$. We would like to see this class of graphs classified.

\begin{problem} Give a combinatorial characterization for squarefree monomial ideals $I$ such that $\reg(I) = 3$. In particular, classify simple graphs $G$ such that $\reg(I(G)) = 3$.
\end{problem}

It was shown in Theorem \ref{thm.vd} that for vertex decomposable simplicial complexes (and hypergraphs) the inductive inequalities in Section \ref{sec.induction} become equalities. It is natural to seek for combinatorial invariants that measure the regularity of these complexes (and hypergraphs).

\begin{problem} Let $H$ be a vertex decomposable hypergraph. Compute $\reg(H)$.
\end{problem}

Examples \ref{ex.smaller} and \ref{ex.ExampleLH} show that inductive inequalities in Theorem \ref{thm.ind11} may become equalities for a larger class of simplicial complexes (and hypergraphs) than vertex decomposable ones. In particular, we would like to classify squarefree monomial ideals for which the conclusion of Theorem \ref{thm.vd} holds.

\begin{problem} Characterize squarefree monomial ideals $I$ for which there exists a variable $x$ such that
$$\reg(I) = \max\{\reg(I:x) + 1, \reg(I,x)\}.$$
\end{problem}

Recall that the inductive bound of Theorem \ref{thm.ind3} comes from the following short exact sequence
\begin{align}
0 \longrightarrow \frac{R}{(x^E) \cap I(H \setminus E)} \longrightarrow \frac{R}{(x^E)} \oplus \frac{R}{I(H \setminus E)} \longrightarrow \frac{R}{I(H)} \longrightarrow 0. \label{eq.sequence}
\end{align}
The bound in Theorem \ref{thm.ind3} would become an equality if the mapping cone construction applied to (\ref{eq.sequence}) results in the \emph{minimal} free resolution of $R/I(H)$. We would like to be able to identify all such edges $E$ for which this is true.

\begin{problem} \label{prob.mapping}
Characterize all edges $E$ in a simple hypergraph $H$ such that the mapping cone construction applied to the short exact sequence (\ref{eq.sequence}) gives the minimal free resolution of $R/I(H)$.
\end{problem}

\cite[Theorem 3.2]{HVT2008} partially answered this problem. Problem \ref{prob.mapping}, in fact, comes from a larger picture, when a squarefree monomial ideal $I$ is splitted into the sum of two sub-ideals. We would also like to see in general which combinatorial structures would enforce the minimality of the mapping cone construction applied to a similar short exact sequence.

\begin{problem} Let $H$ and $H'$ be simple hypergraphs on the same vertex set $X$ with disjoint edge sets $\E(H)$ and $\E(H')$. Let $I = I(H)$ and $J = I(H')$. Find combinatorial conditions for the mapping cone construction applied to the short exact sequence
$$0 \longrightarrow \frac{R}{I \cap J} \longrightarrow \frac{R}{I} \oplus \frac{R}{J} \longrightarrow \frac{R}{I+J} \longrightarrow 0$$
to give the minimal free resolution of $R/I+J$.
\end{problem}

Along the same line, it is also desirable to see when the inductive inequalities in Theorem \ref{thm.ind2} and Corollary \ref{cor.ind2} become equalities.

\begin{problem} Let $H$, $H_1, \dots, H_s$ be simple hypergraphs over the same vertex set $X$, and assume that $\E(H) = \bigcup_{i=1}^s \E(H_i)$. Find combinatorial conditions for the following equality to hold:
$$\reg(R/I(H)) = \sum_{i=1}^s \reg(R/I(H_i)).$$
\end{problem}

An important problem in commutative algebra during the last few decades is to study the regularity of powers of ideals. It is known (cf. \cite{BCH, CHT, Ko, TW}) that for a homogeneous ideal $I \subseteq R$, $\reg(I^q)$ is asymptotically a linear function, i.e., there exist constants $a,b$ and $q_0$ such that for $q \ge q_0$, $\reg(I^q) = aq+b$. While the coefficient $a$ is well-understood, the constants $b$ and $q_0$ are quite mysterious. Many recent studies are devoted to investigating these constants (cf. \cite{Be, Ch, EH, EU, Ha}). Even for edge ideals of graphs the exact values for these constants are still out of reach.

\begin{problem} Let $I$ be a squarefree monomial ideal. Find combinatorial interpretations of the linear function $\reg(I^q)$ for $q \gg 0$.
\end{problem}

For a very simple family of edge ideals, we would like to see constants $b$ and $q_0$ evaluated explicitly.

\begin{problem} Let $I = I(C_k)$ be the edge ideal of a $k$-cycle. Find the linear form of $\reg(I^q)$ for $q \gg 0$.
\end{problem}

Computational experiments lead us to conjecture that for this class of edge ideals, the regularity of $I^q$ becomes a linear function as soon as $q \ge 2$.

\begin{conjecture} Let $I = I(C_k)$ be the edge ideal of a $k$-cycle. Then $\reg(I^q)$ is a linear function for $q \ge 2$.
\end{conjecture}

%%%%%%%%%%%%%%%%%%%%%%%%%%%%%%%%%%%%%%%%%%%%%%%%%%%%%%%%%%%%%%%%%%%%%%%%%%%%%%

\end{document}